\documentclass[12pt,a4paper]{amsart}
\usepackage{mathrsfs}
\usepackage[top=40mm, bottom=40mm, left=35mm, right=35mm]{geometry}

\usepackage{amsfonts}
\usepackage{amsfonts}

\usepackage{epsfig}
\usepackage{amsmath}
\usepackage{amssymb}
\usepackage{amscd}
\usepackage{graphicx}

\usepackage{pstricks}

\usepackage{tocvsec2}

\usepackage[all]{xy}
\usepackage{hyperref}
\newtheorem{thm}{Theorem}[section]

\newtheorem*{thmA}{\bf Theorem A}
\newtheorem*{thmB}{\bf Theorem B}
\newtheorem*{thmC}{\bf Theorem C}

\newtheorem{lem}[thm]{Lemma}

\newtheorem{cor}[thm]{Corollary}

\theoremstyle{definition}
\newtheorem{de}[thm]{Definition}

\theoremstyle{remark}

\numberwithin{equation}{section}

\def \N {\mathbb N}

\def \R {\mathbb R}

\def \N {\mathbb N}

\def \R {\mathbb R}

\numberwithin{equation}{section}

\begin{document}
\title[Time-restricted sensitivity and entropy]{Time-restricted sensitivity and entropy}

\author[K. Liu]{Kairan Liu}
\address{K. Liu: Department of Mathematics, University of Science and Technology of China,
Hefei, Anhui, 230026, P.R. China}
\email{lkr111@mail.ustc.edu.cn}
\author[L. Xu]{Leiye Xu}
\address{L. Xu: Department of Mathematics, University of Science and Technology of China,
Hefei, Anhui, 230026, P.R. China}
\email{leoasa@mail.ustc.edu.cn}
\author[R. Zhang]{Ruifeng Zhang}
\address{R. Zhang: School of Mathematics, Hefei University of Technology,
Hefei, Anhui, 230009, P.R. China}
\email{rfzhang@hfut.edu.cn}
\thanks{K.Liu was supported by NNSF of China (11731003), L.Xu was supported by
	NNSF of China (11801538), R. Zhang was supported by NNSF of China (11871188,11671094).}
\subjclass[2010]{37A35, 54H20}
\keywords{restricted sensitivity, asymptotic rate,entropy}

\date{}

\begin{abstract}
In this paper, we consider measure-theoretical restricted sensitivity and topological restricted sensitivities by restricting the first sensitive time. For a given topological dynamical system, we define measure-theoretical restricted asymptotic rate with respect to sensitivity, and obtain that it equal to the reciprocal of the Brin-Katok local entropy for almost every point. For topological version we have similar definitions and conclusions.
\end{abstract}
\maketitle

\section{Introduction}
%--------------------
Throughout this paper, by {\it a topological dynamical system} (TDS for short) we mean a pair $(X,T)$, where $X$ is a compact metric space with metric $\rho$ and $T:X\rightarrow X$ is a continuous map. The set of all the $T$-invariant Borel probability measures and ergodic measures on $X$ are denoted by $M(X,T)$ and $M^e(X,T)$.

Since the notions were introduced in \cite{R,AY,Gu}, sensitivity has been widely studied as a characterization of chaos for topological dynamical system, which says roughly that there are two points in each non-empty open subset whose trajectories are apart from a given positive constant at some moments. A measure-theoretic version of sensitivity was introduced in \cite{JKL}, and has been further studied by several researchers (\cite{GIL,HLY,HMS,G}).

The notion of sensitivity was generalized by measuring the set of nonnegative integers for which the sensitivity occurs \cite{HKKZ,HKZ,L,LY,TK}. Recently thick sensitivity, thickly syndetical sensitivity where demand those sets be thick or thickly syndetic in  \cite{LLW,TK}. Ye and Yu introduced block (resp. strongly) thickly (resp. IP) sensitivity and proved several Aulander-Yorke's type dichotomy theorems in \cite{YY}.

Notions of measurable sensitivity that restricting the first sensitive time for a standard Lebesgue space was developed in \cite{A.D.F.K.L.S}. The authors showed that these notions are related to positive metric entropy of finite measure-preserving system, and explored the notions in the context of (non-measure-preserving) nonsingular systems.

Our aim in this paper is to investigate measure-theoretical restricted sensitivity and topological restricted sensitivity for a TDS.

Let $(X,T)$ be a TDS. For any $x\in X$ and $\mu\in M(X,T)$, we introduce the notion of measure-theoretical restricted asymptotic rate of $\mu$ at $x$ (see section 3 for details). And we will prove that it has the following relation with Brin-Katok's local entropy at $x$ with respect to $\mu$ (see section 2 for definitions).

\begin{thmA}\label{thm-1}Let $(X,T)$ be a TDS and $\mu \in M(X,T)$. Then for $\mu$-a.e. $x\in X$, one has $$a_\mu(x)=\frac{1}{h_\mu(T,x)}$$
where $a_{\mu}(x)$ is $\mu$-restricted asymptotic rate and $h_\mu(T,x)$ is the Brin-Katok's local entropy at $x$ with respect to $\mu$ and we note: $\frac{1}{0}=+\infty$.
If in addition $\mu$ is ergodic, then $T$ is measure-theoretic restricted sensitive if and only if $h_{\mu}(X,T)>0$.
\end{thmA}

For a TDS $(X,T)$, we can also define topological version of restricted sensitivities. Actually, for any $x\in X$ we introduce type-I and type-II topological restricted asymptotic rates denoted by $a_{1}(x)$ and $a_2(x)$ respectively (see section 4 for definitions). Then we have the following results.

\begin{thmB}\label{thm-2}Let $(X,T)$ be a TDS, $x\in X$ and $a_1(x)$ be the type-I topological restricted asymptotic rate at $x$. Then one has
 $$a_1(x)\le\inf_{\mu\in M(X,T)}a_\mu(x).$$
Moreover, $\inf\limits_{x\in X}a_1(x)\le \frac{1}{h_{top}(T)}$, where $h_{top}(T)$ is the topological entropy of $(X,T)$.
\end{thmB}

\begin{thmC}\label{thm-3} Let $(X,T)$ be a TDS and $a_2(x)$ be the type-II topological restricted asymptotic rate at $x\in X$. Then
 $$\sup_{x\in X}\frac{1}{a_2(x)}=h_{top}(T).$$
\end{thmC}
The paper is organized as follows. In Section 2, we will list some notions and properties needed in our proof. In Section 3, we study the measure-theoretic restricted sensitivity and
prove Theorem~A. Finally, topological version of restricted sensitivities will be introduced in Section 4, where Theorem~B and Theorem~C will be proved.

\section{Preliminaries}
In this section, we will introduce measure-theoretic entropy and topological entropy for a TDS as well as local entropy. Ergodic decomposition and generic points will be also recalled.
\subsection{Entropies for a TDS}
The notion of entropy for a measure-preserving system was introduced by Kolmogorov in \cite{ANK} to describes the complexity.
In \cite{A.K.M}, Alder, Konheim and McAndrew defined the topological entropy for a TDS using open covers as an invariant of topological conjugacy and also as an analogue of measure theoretic entropy. There are equivalent definitions of topological entropy using separating and spanning sets, this was done by Dinaburg and by Bowen.

Let $(X,T)$ be a TDS and $K$ be a compact subset of $X$. For any $\epsilon >0$ and $n\in\N$, we call a set $R\subset X$ an {\it $(n,\epsilon)$-spanning set for K}, if for every $x\in K$ there exists $y\in R$, such that $\rho_n(x,y)<\epsilon$, where the compatible metric  $\rho_n$ is defined by the formula
$$\rho_n(x,y)=\max_{0\leqslant i \leqslant n}\rho(T^ix,T^iy).$$
Let $r(n,\epsilon,K)\in \mathbb{N}$ denote the smallest cardinality of all $(n,\epsilon)$-spanning sets for $K$. We call a finite set $S \subset K$ an {\it $(n,\epsilon)$-separated set for K}, if for every distant $x,y \in S$ we have $d_n(x,y)\geq \epsilon$. We denote by $s(n,\epsilon,K)\in \mathbb{N}$ the largest cardinality of all $(n,\epsilon)$-separated set for $K$. When $K=X$ we omit the restriction of $K$. Then the topological entropy of $T$ on $X$ is given by:
\begin{equation}\nonumber
\begin{aligned}h_{top}(T)&=\lim_{\epsilon\to0}\liminf_{n\to\infty}\frac{1}{n}\log(r(n,\epsilon))=
\lim_{\epsilon\to0}\limsup_{n\to\infty}\frac{1}{n}\log(r(n,\epsilon))\\
&=\lim_{\epsilon\to0}\liminf_{n\to\infty}\frac{1}{n}\log(s(n,\epsilon)) = \lim_{\epsilon\to0}\limsup_{n\to\infty}\frac{1}{n}\log(s(n,\epsilon))\\
\end{aligned}
\end{equation}

For a TDS $(x,T)$, Goodwyn, Dinaburg and Goodman proved the basic relationship between topological entropy and measure-theoretic entropy, the variational principle (\cite{LWG,TG,D}),
$$h_{top}(T)=\sup\limits_{\mu\in M(X,T)}h_\mu(T).$$

\subsection{Local entropy}
For a TDS $(X,T)$ and $\mu\in M(X,T)$, Brin and Katok proved that for $\mu$-a.e. $x\in X$, the following equation holds ( see\cite{Br.Ka}):
\begin{align*}
\lim_{\epsilon \to0}\liminf_{n\to \infty}\frac{-\log \mu(B_n(x,\epsilon))}{n}
=\lim_{\epsilon \to0}\limsup_{n\to \infty}\frac{-\log \mu(B_n(x,\epsilon))}{n},
\end{align*}
where $B_n(x,\epsilon)=\{y \in X: \rho_n(x,y)<\epsilon\}$.
And the common value are defined as $h_{\mu}(T,x)$, the {\it local entropy at $x$ with respect to $\mu$}. They also show that following relation between $h_{\mu}(T,x)$ and $h_{\mu}(T)$
$$h_{\mu}(T)=\int h_{\mu}(T,x)d\mu(x).$$
Especially, in the case $\mu$ is ergodic, $h_{\mu}(T,x)=h_\mu(T)$ for $\mu$ $x\in X$.

\subsection{Ergodic decomposition and generic points}
For a TDS $(X,T)$, an important consequence of the fact that $M(X,T)$ is a compact convex set is that it gives a way to decompose any invariant measure into ergodic components (see \cite[chapter 4]{Thomas} for details).
\begin{thm}\label{ergodic decomposition}
Let $(X,T)$ be a TDS. Then for any $\mu\in M(X,T)$ there is a unique probability measure $\lambda$ defined on $M(X,T)$ with the properties that
\begin{enumerate}
\item $\lambda(M^e(X,T))=1$.
\item $\int_X f d\mu=\int_{M^e(X,T)}(\int_X fd\nu) d\lambda(\nu)$.
\end{enumerate}
\end{thm}

Let $(X,T)$ be a TDS, $\mu\in M(X,T)$. We say $x_0\in X$ is a {\it generic point} for $\mu$ if
$$\frac{1}{n}\sum_{i=1}^{n-1}f(T^ix_0)\to\int fd\mu,$$
for every continuous function $f\in C(X)$. By Birkhoff ergodic theorem, we know that if $\mu\in M^e(X,T)$ then $\mu$-almost every $x\in X$ is a generic point for $\mu$.

\section{measure-theoretical restricted sensitivity }
In this section, we introduce the notions of measure-theoretic restricted sensitivity and measure-theoretic restricted asymptotic rate, then study their relations with measure-theoretic entropy. For a TDS $(X,T)$ and $\mu\in M(X,T)$, we set $\mathscr{B}_{\mu}^+=\{V\in\mathscr{B}_X,\ \mu(V)>0\}$, where $\mathscr{B}_X$ is the $\sigma$-algebra of Borel subsets of $X$.

\begin{de}\label{def-m-sens}Let $(X,T)$ be a TDS and $\mu\in M(X,T)$. Then for any $x\in X$ and $\delta>0$,  we say $T$ is {\it $\mu$-restricted sensitive at $x$ w.r.t. $\delta$} if there exists $a>0$, such that for every $V\in\mathscr{B}_{\mu}^+$ there exists nonnegative integer $n \le -a\log\mu(V)$ such that
$$\mu(\{\ y\in V: \rho(T^nx,T^ny)>\delta \}\ )>0.$$
$T$ is called {\it $\mu$-restricted sensitive} if for $\mu$-a.e. $x\in X$, there exists $\delta>0$ such that $T$ is measure-theoretic restricted sensitive at $x$ w.r.t. $\delta$.
\end{de}
We call $a$ in the definition \ref{def-m-sens} is a {\it $\mu$-asymptotic rate at $x$ w.r.t $\delta$}.
We further consider {\it $a_\mu(x,\delta)$, $\mu$-restricted asymptotic rate at $x$ with respect to $\delta$}, which is defined to be the infimum over all $\mu$-asymptotic rate at $x$ w.r.t $\delta$. If $T$ is not $\mu$-restricted sensitive at $x$ w.r.t. $\delta$, then we let $a_\mu(x,\delta)=+\infty$.

From another point of view, we have another characterization of measure-theoretic restricted asymptotic rate. Let $(X,T)$ be a TDS and $\mu\in M(X,T)$. Then for any $x\in X$, $\delta>0$ and $V\in\mathscr{B}^+_{\mu}$, we set $s_{\mu}(x,V,\delta)$ be the minimal nonnegative integer $n$ (if such number exists) such that
$$\mu(\{\ y\in V: \rho(T^nx,T^ny)>\delta \}\ )>0.$$
We let $s_{\mu}(x,V,\delta)=+\infty$ if such number does not exist. Then we have the following lemma.
\begin{lem}\label{3.2}
Let $(X,T)$ be a TDS and $\mu\in M(X,T)$. Then for any $x\in X$ and $\delta>0$ we have
$$\frac{1}{a_{\mu}(x,\delta)}=\inf_{V\in\mathscr{B}_{\mu}^+}\frac{-\log\mu(V)}{s_{\mu}(x,V,\delta)}$$
where for any constant $c\ge 0$ we note $\frac{c}{0}=+\infty$.
\end{lem}
\begin{proof}For any fixed $x\in X$ and $\delta>0$, there are two cases. In the case that $T$ is not $\mu$-restricted sensitive at $x$ w.r.t. $\delta$, i.e. $a_{\mu}(x,\delta)=+\infty$, one has that for every $a\in (0,+\infty)$ there exists a subset $V_a\in\mathscr{B}_{\mu}^+$, such that for any $n\leq-a\log{\mu(V_a)}$ one has
$$\mu(\{y\in V_a: \rho(T^nx,T^ny)>0\})=0.$$
This means $s_{\mu}(x,V_a,\delta)>-a\log{\mu(V_a)}$. Thus $\frac{1}{a}>\inf\limits_{V\in\mathscr{B}_{\mu}^+}-\frac{\log{\mu(V)}}{s_{\mu}(x,V,\delta)}$, let $a\to+\infty$ we obtain the conclusion for the case $a_{\mu}(x,\delta)=+\infty$.

Now we assume $T$ is $\mu$-restricted sensitive at $x$ w.r.t. $\delta$, this is $a_{\mu}(x,\delta)<+\infty$. Then for any $a>a_{\mu}(x,\delta)$ and $V\in\mathscr{B}_{\mu}^+$ there exists a nonnegative integer $n \le -a\log\mu(V)$ such that
$$\mu(\{\ y\in V: \rho(T^nx,T^ny)>\delta \}\ )>0.$$
Hence $s_{\mu}(x,V,\delta) \le  -a\log\mu(V)$, that is, $\frac{1}{a}\le\frac{-\log\mu(V)}{s_{\mu}(x,V,\delta)}$ for any $V\in\mathscr{B}_{\mu}^+$. Then let $a \to a_{\mu}(x,\delta)$, we have
\begin{equation}\label{eq-a-mu-s}
  \frac{1}{a_{\mu}(x,\delta)}\leq \inf_{V\in\mathscr{B}_{\mu}^+}\frac{-\log\mu(V)}{s_{\mu}(x,V,\delta)}.
\end{equation}

On the other hand, for any $a<a_{\mu}(x,\delta)$ there exists $V_a\in\mathscr{B}_{\mu}^+$ such that for any $n\leq -a\log{\mu(V_a)}$ we have $\mu(\{y\in V_a: \rho(T^{n}x,T^ny)>\delta\})=0$. This implies $s_{\mu}(x,V_a,\delta)>-a\log{\mu(V_a)}$, thus $\frac{1}{a}>\inf\limits_{V\in\mathscr{B}_{\mu}^+}\frac{-\log{\mu(V)}}{s_{\mu}(x,V,\delta)}$. Let $a\to a_{\mu}(x,\delta)$ we have
\begin{equation}\label{eq-a-mu}
  \frac{1}{a_{\mu}(x,\delta)}\geq \inf_{V\in\mathscr{B}_{\mu}^+}\frac{-\log\mu(V)}{s_{\mu}(x,V,\delta)}.
\end{equation}

Combining \eqref{eq-a-mu-s} with \eqref{eq-a-mu}, we obtain the conclusion for the case $a_{\mu}(x,\delta)<+\infty$. And this ends our proof.
\end{proof}

Notice that $a_{\mu}(x,\delta)$ no increases as $\delta$ decreases. We can define
 {\it $\mu$-restricted asymptotic rate at $x$} by
$$a_{\mu}(x)=\lim_{\delta\rightarrow 0}a_{\mu}(x,\delta),$$
and we let $a_{\mu}(x)=+\infty$ if $a_{\mu}(x,\delta)=+\infty$ for any $\delta>0$. We have the following observations.
\begin{lem}\label{lem-1} Let $(X,T)$ be a TDS and $\mu\in M(X,T)$. Then for every $x\in X$ one has
\begin{enumerate}
  \item $x \notin supp(\mu)$ implies $a_{\mu}(x)=0$;
  \item $\mu(\{x\})>0$ implies $a_{\mu}(x)=+\infty$.
\end{enumerate}
\end{lem}
\begin{proof}
$(1)$. For any $x \notin supp(\mu)$, we can find $\delta_0>0$ such that $\mu(B(x,2\delta_0))=0$. Notice that $a_{\mu}(x,\delta)$ is non negative and no increases as $\delta$ decreases. To show $a_{\mu}(x)=0$, it is enough to show that $a_{\mu}(x,\delta_0)=0$.

For any subset $V\in\mathscr{B}_\mu^+$, we set $\widetilde{V}=V \setminus B(x,2\delta_0)$. Then $\mu(V)=\mu(\widetilde V)>0$ and $\rho(x,y)>\delta_0$ for every $y\in\widetilde V$. That is, for any $a>0$ there is $0\le a\log(\mu(\widetilde{V}))$, such that
$$\mu(\{y\in V: \rho(x,y)>\delta_0\})=\mu(\widetilde V)>0.$$
This implies $a_{\mu}(x,\delta_0)=0$, and this ends the proof of $(1)$.

$(2)$. Assume $\mu(\{x\})>0$. We set $V=\{x\}$, then $\mu(V)>0$. For any $\delta>0$ and $n \in \mathbb{Z}_+$, one has
$$\mu(\{y\in V: \rho(T^nx,T^ny)>\delta\})=\mu(\emptyset)=0.$$
This is $a_\mu(x,\delta)=+\infty$ by definition. By taking $\delta\to 0$ we have $a_{\mu}(x)=+\infty$, and this finishes the proof of $(2)$.
\end{proof}

To prove Theorem A, we need the following lemmas.
\begin{lem}\label{333}Let $(X,T)$ be a TDS and $\mu\in M(X,T)$. Then for every $x\in X$ we have
\begin{equation}\label{eq-3}
\lim_{\delta\to0}\inf_{n\ge 1}\frac{-\log\mu(B_n(x,\delta))}{n}
=\lim_{\delta\to0}\liminf_{n\to+\infty}\frac{-\log\mu(B_n(x,\delta))}{n}
\end{equation}
\end{lem}
\begin{proof}For any fixed $x\in X$, it is clear that
$$\lim_{\delta\to 0}\inf_{n\ge1}\frac{-\log\mu(B_n(x,\delta))}{n}
\le \lim_{\delta\to0}\liminf_{n\to+\infty}\frac{-\log\mu(B_n(x,\delta))}{n}.$$

Now we shall to show
$$\lim_{\delta\to0}\liminf_{n\to+\infty}\frac{-\log\mu(B_n(x,\delta))}{n}\le\lim_{\delta\to 0}\inf_{n\ge1}\frac{-\log\mu(B_n(x,\delta))}{n}.$$
We have two cases. If $c:=\lim\limits_{\delta\to0}\liminf\limits_{n\to+\infty}\frac{-\log\mu(B_n(x,\delta))}{n}=0$, there is nothing to prove. Now we assume $c>0$. If the equality \eqref{eq-3} does not hold, this is,
$$\lim_{\delta \rightarrow 0} \inf_{n \ge 1}\frac{-\log \mu(B_n(x,\delta))}{n}
< \lim_{\delta \rightarrow 0} \liminf_{n \rightarrow \infty}\frac{-\log \mu(B_n(x,\delta))}{n}.$$
Then there exists $0<\varepsilon<c$ such that
\begin{align}\label{eq4}
\lim_{\delta\to0}\inf_{n\ge 1}\frac{-\log \mu(B_n(x,\delta))}{n}<c-\varepsilon,
\end{align}
and $\delta_0>0$ such that fore every $0<\delta\le\delta_0$ one has
$$\liminf_{n\to+\infty}\frac{-\log \mu(B_n(x,\delta))}{n}>c-\frac{\varepsilon}{2}.$$
Then there exists $N_0\in\N$ such that $\frac{-\log\mu(B_m(x,\delta_0))}{m}>c-\frac{\varepsilon}{2}$ for every $m\ge N_0$. For any $0<\delta\le\delta_0$ and $m\ge N_0$, since $-\log\mu(B_m(x,\delta_0))\le-\log\mu(B_m(x,\delta))$, one has
\begin{align}\label{eq5}
\frac{-\log\mu(B_m(x,\delta))}{m}>c-\frac{\varepsilon}{2}.
\end{align}
By (\ref{eq4}), there exists $0<\delta_1<\delta_0$ such that for every $0<\delta\le\delta_1$ we have $\inf\limits_{n\geq 1}\frac{-\log\mu(B_n(x,\delta))}{n}<c-\varepsilon$. Combining this with (\ref{eq5}), one has that for every $0<\delta\le\delta_1$
$$\inf_{n\ge 1}\frac{-\log\mu(B_n(x,\delta))}{n}=\min_{1\le n\le N_0}\frac{-\log\mu(B_n(x,\delta))}{n}< c-\varepsilon.$$
Let $\delta\to0$, we have $\min\limits_{1\le n\le N_0}\frac{-\log\mu(\{x\})}{n}\le c-\varepsilon$, this implies $\mu(\{x\})>0$. Then we have the following observation
$$c=\lim_{\delta\to0}\liminf_{n\to+\infty}\frac{-\log\mu(B_n(x,\delta))}{n}\le\liminf_{n\to+\infty}\frac{-\log\mu(\{x\})}{n}=0,$$
which is contradicted with $c>0$. Thus the equality \eqref{eq-3} holds.
\end{proof}

For any $t\in\R^+$, we set $\lfloor t\rfloor=\max\{n\in\N, n\le t\}$. Then we have the following result.
\begin{lem}\label{lm1}Let $(X,T)$ be a TDS and $\mu\in M(X,T)$. Then for every $x\in X$ we have
$$\frac{1}{a_\mu(x)}= \lim_{\delta\to0}\inf_{n\ge 1}\frac{-\log\mu(B_n(x,\delta))}{n},$$
where for any constant $c\ge 0$ we note $\frac{c}{0}=+\infty$.
\end{lem}
\begin{proof}For every fixed $x\in X$, there are two cases. In the case $\mu(\{x\})>0$, we have $a_\mu(x)=+\infty$ by  Lemma \ref{lem-1} (2).
Then $$0\le \lim_{\delta\to0}\inf_{n\ge1}\frac{-\log\mu(B_n(x,\delta))}{n}\le \lim_{\delta\to0}\inf_{n\ge1}\frac{-\log\mu(\{x\})}{n}=0= \frac{1}{a_\mu(x)}.$$
Hence, $\lim\limits_{\delta\to0}\inf\limits_{n\ge1}\frac{-\log\mu(B_n(x,\delta))}{n}= \frac{1}{a_\mu(x)}$ when $\mu(\{x\})>0$.

Now we assume $\mu(\{x\})=0$. Firstly, we show that
\begin{align}\label{eq-4}\frac{1}{a_\mu(x)}\ge\lim_{\delta\to 0}\inf_{n\ge 1}\frac{-\log\mu(B_n(x,\delta))}{n}.\end{align}
If $a_\mu(x)=0$, there is nothing to prove.
We assume that $a_\mu(x)>0$. For any $a\in(0,a_\mu(x))$ there exists $\delta>0$ such that $a<a_{\mu}(x,\delta)$. Thus there exists $V_{\delta}\in\mathscr{B}_{\mu}^+$ such that
$$\mu(\{y\in V_\delta:d(T^nx,T^ny)>\delta\})=0$$
for all $n\le-a\log\mu(V_\delta)$. Then
$$V_\delta\subseteq \overline{B_{\lfloor-a\log\mu(V_\delta)\rfloor}(x,\delta)}\subseteq B_{\lfloor-a\log\mu(V_\delta)\rfloor}(x,2\delta)\mod\mu.$$
Hence
\begin{equation*}
\begin{aligned}
\frac{-\log\mu(V_\delta)}{\lfloor-a\log\mu(V_\delta)\rfloor}
&\ge&\frac{-\log\mu(B_{\lfloor-a\log\mu(V_\delta)\rfloor}(x,2\delta))}{\lfloor-a\log\mu(V_\delta)\rfloor}\\
&\ge& \inf_{n\ge1}\frac{-\log\mu(B_n(x,2\delta))}{n},
\end{aligned}
\end{equation*}
Notice that, $V_{\delta}\subseteq B(x,2\delta)$ then $\mu(V_\delta)\to0$ as $\delta\to 0$ since $\mu(\{x\})=0$. We have $\lfloor-a\log\mu(V_\delta)\rfloor\to+\infty$ as $\delta\to 0$. Then
$$\frac{1}{a}=\lim_{\delta\to0}\frac{-\log\mu(V_\delta)}{\lfloor-a\log\mu(V_\delta)\rfloor}\ge\lim_{\delta\to0}\inf_{n\ge 1}\frac{-\log\mu(B_n(x,\delta))}{n}.$$
By taking $a\to a_\mu(x)$, we have $\frac{1}{a_\mu(x)}\ge\lim\limits_{\delta\to0}
\inf\limits_{n\ge 1}\frac{-\log\mu(B_n(x,\delta))}{n}$ .

Next, we shall show $\frac{1}{a_\mu(x)}\le\lim\limits_{\delta\to0}\inf\limits_{n\ge 1}\frac{-\log\mu(B_n(x,\delta))}{n}.$ If $a_\mu(x)=+\infty$, it is obviously true. In the case $a_\mu(x)<+\infty$, if there exists $\delta_0>0$ and $n_0\in\N$ such that $\mu(B_{n_0}(x,\delta_0))=0$, then for any $0<\delta<\delta_0$ and $n\ge n_0$ we have $\mu(B_{n}(x,\delta))=0$, thus by Lemma \ref{333}
$$\lim_{\delta\to0}\inf_{n\ge 1}\frac{-\log\mu(B_n(x,\delta))}{n}=\lim_{\delta \to 0}\liminf_{n\to +\infty}\frac{-\log\mu(B_n(x,\delta))}{n}=+\infty\ge\frac{1}{a_{\mu}(x)}.$$

If for any $\delta>0$ and $n\in\N$, $\mu(B_n(x,\delta))>0$, one has $s_{\mu}(x,B_{n}(x,\delta),\delta)>n$, combining with Lemma \ref{3.2}, one has
\begin{align*}
\inf_{n\ge 1}\frac{-\log\mu(B_n(x,\delta))}{n}&\ge \inf_{n\ge 1}\frac{-\log\mu(B_n(x,\delta))}{s_{\mu}(x,B_n(x,\delta),\delta)}\\
&\ge\inf_{V\in\mathscr{B}_{\mu}^+}\frac{-\log\mu(V)}{s_{\mu}(x,V,\delta)}=\frac{1}{a_{\mu}(x,\delta)}.
\end{align*}
By taking $\delta\to 0$, one has
$$\lim_{\delta\to0}\inf_{n\ge 1}\frac{-\log\mu(B_n(x,\delta))}{n}\ge\frac{1}{a_{\mu}(x)}.$$
Combing this with \eqref{eq-4}, we prove the conclusion in the case $\mu(\{x\})=0$, and this completes our proof.
\end{proof}

Now we are ready to prove Theorem A.
\begin{proof}[Proof of Theorem A] We fix $\mu\in M(X,T)$. Notice that for $\mu$-a.e. $x\in X$ $$h_\mu(T,x)=\lim_{\delta\to0}\liminf_{n\to\infty}\frac{-\log\mu(B_n(x,\delta))}{n}.$$
Then by Lemma \ref{333} and Lemma \ref{lm1}, we have for $\mu$-a.e. $x\in X$
\begin{equation}\nonumber
\begin{aligned}\frac{1}{a_\mu(x)}&=\lim_{\delta\to0}\inf_{n\ge 1}\frac{-\log\mu(B_n(x,\delta))}{n}\\
&= \lim_{\delta\to0}\liminf_{n\to\infty}\frac{-\log\mu(B_n(x,\delta))}{n}=h_\mu(T,x).
\end{aligned}
\end{equation}

Now we assume that $\mu$ is ergodic. In this case, the Brin-Katok's entropy formula gives that $h_\mu(T)=h_\mu(T,x)$ for $\mu$-a.e. $x\in X$. Combing with the above discussion, we have $a_\mu(x)=\frac{1}{h_\mu(T)}$ for $\mu$-a.e. $x\in X$. In this case, $T$ is measure-theoretic restricted sensitive if and only if $h_{\mu}(X,T)>0$. This ends the proof of Theorem \ref{thm-1}.
\end{proof}

\section{topological restricted sensitivity}
In this section, we consider the topological version of restricted sensitivity. We will introduce two types of topological restricted sensitivities and two types of restricted asymptotic rates respectively, and we study their relations with measure-theoretic restricted asymptotic rate and topological entropy.

\subsection{Type-I Topological Restricted Sensitivity}\label{subsec-type1}
Let $(X,T)$ be a TDS, for any subset $V$ of $X$, we define {\it the capacity of $V$} by
$$c(V)=\inf_{x\in X}\limsup_{n\to\infty}\frac{1}{n}\sum_{i=0}^{n-1}\chi_V(T^ix).$$
where $\chi_V(x)=1$ if and only if $x\in V$. We set
$$OC(X,T):=\{V\subseteq X,\ V\ is\ open\ and\ c(V)>0\}.$$
Note that, $OC(X,T)$ is not empty since $X\in OC(X,T)$. The capacity and measures of a subset of $X$ have the following relationship.
\begin{lem}\label{capacity}Let $(X,T)$ be a TDS. Then for any open subset $V$ of $X$ one has
$$c(V)=\inf\limits_{\mu\in M(X,T)}\mu(V).$$
\end{lem}
\begin{proof}Note that, for any $x\in X$ one has $\nu:=\limsup\limits_{n\to\infty}\frac{1}{n}\sum\limits_{i=0}^{n-1}\delta_{T^ix}\in M(X,T)$. Thus for any subset $V$ of $X$,
$$\inf_{\mu\in M(X,T)}\mu(V)\le\nu(V)=\limsup\limits_{n\to\infty}\frac{1}{n}\sum_{i=0}^{n-1}\delta_{T^ix}(V)=
\limsup\limits_{n\to\infty}\frac{1}{n}\sum_{i=0}^{n-1}\chi_V(T^ix)$$
for every $x\in X$. This implies $c(V)\ge \inf\limits_{\mu\in M(X,T)}\mu(V)$.

Moreover, for any $\mu\in M(X,T)$, by Theorem \ref{ergodic decomposition}, there exists $\nu\in M^e(X,T)$ such that $\nu(V)\le\mu(V)$. Since $\nu$ is ergodic, there exists $x_0\in X$ which is a generic point for $\mu$. This implies that for any open subset $V$, one has
$$\limsup\limits_{n\to\infty}\frac{1}{n}\sum_{i=0}^{n-1}\chi_V(T^ix_0)=\int\chi_V(x)d\nu(x)=\nu(V).$$
Then we have $\mu(V)\ge c(V)$ for any $\mu\in M(X,T)$. Thus, $c(V)\le\inf\limits_{\mu\in M(X,T)}\mu(V)$. To summing up, one has
 $$c(V)=\inf\limits_{\mu\in M(X,T)}\mu(V)$$
for every open subset of $X$.
\end{proof}

Then the topological  restricted sensitivity is defined as follows.
\begin{de}\label{type1}Let $(X,T)$ be a TDS. For any $x \in X$ and $\delta\in(0,\infty)$, $T$ is called {\it type-I topological restricted sensitive (TRS-I for short) at $x$ w.r.t. $\delta$} if there exists $a>0$ such that for every subset $V\in OC(X,T)$, there exists nonnegative integer $n \le -a\log c(V)$ and $y\in V$ such that
$$\rho(T^nx,T^ny)>\delta.$$
And we say $T$ is {\it TRS-I} if there exists a dense subset $D\subset X$ such that for any $x \in D$ there exists $\delta(x)>0$ such that $T$ is TRS-I at
$x$ w.r.t $\delta(x)$.
\end{de}

For any $\delta>0$ and $x\in X$. If $T$ is  TRS-I at $x$ w.r.t. $\delta$, we further consider $a_{1}(x,\delta)$, the {\it topological restricted asymptotic rate at $x$ with respect to $\delta$}, which is defined to be the infimum over all $a>0$ satisfies the property in Definition \ref{type1}. If $T$ is not TRS-I at $x$ w.r.t. $\delta$, then let $a_{1}(x,\delta)=+\infty$.

Similar as in measure-theoretic case. We can also see $a_{1}(x,\delta)$ from another point of view.
Let $\delta>0$ and $V\in OC(X,T)$, we denote the set of sensitive time with respect to the sensitivity constant $\delta$ by
$$\tilde{S}(x,V,\delta)=\{n\ge0:\text{there exists } y\in V\text{ such that } \rho(T^nx,T^ny)>\delta\}.$$
 The first sensitive time  is denoted by
\[
\tilde{s}(x,V,\delta) = \left\{
\begin{array}{ll}
\min \{n\in \tilde{S}(x,V,\delta)\}, & \text{if } \tilde{S}(x,V,\delta)\neq \phi,\\
\\
+\infty, & \text{if } \tilde{S}(x,V,\delta)=\phi.\\
\end{array}
\right.
\]
We have the following observation.
\begin{lem}\label{tp.re.eq}
Let $(X,T)$ be a TDS. Then for any $x \in X$ and $\delta>0$ one has
$$\frac{1}{a_1(x,\delta)}=\inf_{V\in OC(X,T)}\frac{-\log(c(V))}{\tilde{s}(x,V,\delta)},$$
where for any constant $c\ge0$ we note $\frac{c}{0}=+\infty$.
\end{lem}
\begin{proof}
We fix any $x\in X$ and $\delta>0$. Firstly, we shall show
\begin{equation}\label{eq-a-1-ss}
\frac{1}{a_{1}(x,\delta)}\ge\inf_{V\in OC(X,T)}\frac{-\log(c(V))}{\tilde{s}(x,V,\delta)}.
\end{equation}
If $a_{1}(x,\delta)=0$ there is nothing to prove. Now we assume $a_{1}(x,\delta)>0$, for any $0<a<a_{1}(x,\delta)$, there exists an open subset $V_a$ of $X$ with positive capacity, such that every $n\le -a\log{c(V_a)}$ and $y\in V$ one has $\rho(T^nx, T^ny)\leq\delta$. This implies $\tilde{s}(x,V_a,\delta)>-a\log c(V_a)$, thus
$$\inf_{V\in OC(X,T)}\frac{-\log{c(V)}}{\tilde{s}(x,V,\delta)}\le\frac{-\log{c(V_a)}}{\tilde{s}(x,V_a,\delta)}<\frac{1}{a}.$$
By taking $a\to a_{1}(x,\delta)$, we can get (\ref{eq-a-1-ss}).

Now we shall show
\begin{equation}\label{eq-a-1-s}
   \frac{1}{a_{1}(x,\delta)}\le\inf_{V\in OC(X,T)}\frac{-\log(c(V))}{\tilde{s}(x,V,\delta)}.
\end{equation}
Since it is obvious when $a_{1}(x,\delta)=+\infty$, we assume $a_{1}(x,\delta)<+\infty$. For any $a>a_{1}(x,\delta)$ and open set $V$ with $c(V)>0$, there exists a nonnegative integer $n \le -a\log(c(V))$ such that
$$\{\ y\in V: \rho(T^nx,T^ny)>\delta \} \neq \emptyset.$$
Hence $\tilde{s}(x,V,\delta) \le  -a\log c(V)$ for any open subset $V$ with positive capacity. If $\tilde{s}(x,V,\delta)>0$, then let $a\to a_1(x,\delta)$ we know \ref{eq-a-1-s} holds. And this ends the proof.
\end{proof}

Note that $a_{1}(x,\delta)$ decrease as $\delta \to 0$, then the {\it type-I topological restricted asymptotic rate at $x$} is defined by
 $$a_1(x)=\lim_{\delta\to0}a_1(x,\delta),$$
and let $a_{1}(x)=+\infty$ if for any $\delta>0$, $a_{1}(x,\delta)=+\infty$. The following Theorem shows the relationship between $a_\mu(x)$ and $a_1(x)$.
\begin{thmB}\label{thm-2}Let $(X,T)$ be a TDS and $x \in X$. Then one has
 $$a_1(x)\le\inf_{\mu\in M(X,T)}a_\mu(x).$$
Moreover, $\inf\limits_{x\in X}a_1(x)\le \frac{1}{h_{top}(T)}$, where $h_{top}(T)$ is the topological entropy of $(X,T)$.
\end{thmB}
\begin{proof}For any $x\in X$, by Lemma \ref{tp.re.eq} and Lemma \ref{capacity} and note that $\tilde{s}(x,V,\delta)\le s_{\mu}(x,V,\delta)$, $OC(X,T)\subseteq \mathscr{B}_{\mu}^+$ for every $\mu\in M(X,T)$, one has
\begin{equation}\nonumber
\begin{aligned}
\frac{1}{a_1(x,\delta)}&=\inf_{V\in OC(X,T)}\frac{-\log(c(V))}{\tilde{s}(x,V,\delta)}\\
&\ge\inf_{V\in OC(X,T)}\sup_{\mu\in M(X,T)}\frac{-\log\mu(V)}{s_{\mu}(x,V,\delta)}
\ge\sup_{\mu\in M(X,T)}\inf_{V\in OC(X,T)}\frac{-\log\mu(V)}{s_{\mu}(x,V,\delta)}\\
&\ge\sup_{\mu\in M(X,T)}\inf_{V\in \mathscr{B}_{\mu}^+}\frac{-\log\mu(V)}{s_{\mu}(x,V,\delta)}
=\sup_{\mu\in M(X,T)}\frac{1}{a_\mu(x,\delta)}.
\end{aligned}
\end{equation}
Hence, $a_1(x,\delta)\le\inf\limits_{\mu\in\mathcal{M}(X,T)}a_\mu(x,\delta)$. Let $\delta\to 0$, one has
$a_1(x)\le\inf\limits_{\mu\in M(X,T)}a_\mu(x)$.
Then
$$\frac{1}{\inf\limits_{x\in X}a_1(x)}=\sup_{x\in X}\frac{1}{a_1(x)} \ge \sup_{x\in X}\sup_{\mu \in M(X,T)}\frac{1}{a_{\mu}(x)},$$
which means for any $x_0\in X$, $\mu \in M(X,T)$, we have $\frac{1}{\inf\limits_{x\in X}a_1(x)}\ge \frac{1}{a_{\mu}(x_0)}$.
Combing this with Theorem~A and the Brin-Katok entropy formula, one has
$$\frac{1}{\inf\limits_{x\in X}a_1(x)}\ge \int_{X}\frac{1}{a_{\mu}(x)}d\mu(x)=\int_{X}h_{\mu}(T,x)d\mu(x)=h_{\mu}(T),$$
for every $\mu\in M(X,T)$. Thus by the variational principle, we have $\frac{1}{\inf\limits_{x\in X}a_1(x)}\ge h_{top}(T)$, and this ends our proof.
\end{proof}

\subsection{Type-II Topological Restricted Sensitivity }\label{subsec-type2}
\begin{de}\label{type2}
Let $(X,T)$ be a TDS. For $x\in X$ and $\delta>0$, $T$ is said to be {\it type-II topological restricted sensitive(TRS-II for short) at $x$ w.r.t. $\delta$} if there exists $a>0$ such that for $\epsilon>0$ small enough there exists $N_{\varepsilon}\in\N$, such that for each Bowen ball $B_N(x,\epsilon)$ with $N>N_{\varepsilon}$, there exists nonnegative integer $n\le a\log r(N,\epsilon)$ and $y\in B_N(x,\epsilon)$ with $$\rho(T^nx,T^ny)>\delta.$$
And $T$ is {\it TRS-II} if there exists a dense subset $D\subset X$ such that for any $x \in D$ there exists $\delta>0$ such that $T$ is TRS-II at
$x$ w.r.t $\delta$.
\end{de}

Let $\delta>0$ and $x\in X$. If $T$ is TRS-II at $x$ w.r.t. $\delta$, we further consider $a_{2}(x,\delta)$, the {\it type-II topological restricted asymptotic rate at $x$ with respect to $\delta$}, which is defined to be the infimum over all $a>0$ satisfies the property in Definition \ref{type2}. If $T$ is not TRS-II at $x$ w.r.t. $\delta$, then let $a_{2}(x,\delta)=+\infty$.

Note that, $a_{2}(x,\delta)$ no increase as $\delta$ decrease, then the {\it type-II topological restricted asymptotic rate at $x$} is defined by
 $$a_2(x)=\lim_{\delta\to0}a_2(x,\delta),$$
and let $a_{2}(x)=+\infty$ if for any $\delta>0$, $a_{2}(x,\delta)=+\infty$.

The following result shows the relation between $a_2(x)$ and the topological entropy.
\begin{thmC}
Let $(X,T)$ be a TDS. Then
 $$\sup_{x\in X}\frac{1}{a_2(x)}=h_{top}(T).$$
\end{thmC}
\begin{proof}Firstly, we show $\sup_{x\in X}\frac{1}{a_2(x)}\le h_{top}(T)$.
For any fixed $x\in X$, since the inequality is obviously true if $a_2(x)=+\infty$, we assume $a_2(x)<+\infty$. Then there exists $\delta(x)>0$ such that $T$ is TRS-II at $x$ w.r.t. $\delta (x)$. For any $a>a_2(x,\delta(x))$ and $0<\epsilon<\delta(x)$ small enough there exists $N_{\epsilon}\in\N$ such that for any $N > N_{\epsilon}$, there is integer $0\le n\le a\log r(N,\epsilon)$ and $y\in B_N(x,\epsilon)$ such that $\rho(T^nx,T^ny)>\delta(x)$. Hence, $B_N(x,\epsilon)\nsubseteq B_{\lfloor a\log r(N,\epsilon)\rfloor}(x,\delta(x))$ for any $N>N_{\epsilon}$. Since $\epsilon<\delta(x)$, one has
$$B_N(x,\delta(x))\nsubseteq B_{\lfloor a\log r(N,\epsilon)\rfloor}(x,\delta(x)).$$
Therefore, $N<\lfloor a\log s_N(T,\epsilon)\rfloor\le a\log r(N,\epsilon)$ for every $N>N_{\epsilon}$. This is, $\frac{1}{a}<\frac{\log r(N,\epsilon)}{N}$ for any $N>N_{\epsilon}$. Taking $a\to a_2(x,\delta(x))$, we have
$$\frac{1}{a_2(x,\delta (x))}\le\lim_{\epsilon\to0}\limsup_{N\to+\infty}\frac{\log r(N,\epsilon)}{N}=h_{top}(T).$$
Finally, let $\delta(x)\to 0$, we have $\frac{1}{a_2(x)}\le h_{top}(T)$ for any $x\in X$. Thus $\sup\limits_{x\in X}\frac{1}{a_2(x)}\le h_{top}(T)$.

Nextly, we shall show $\sup\limits_{x\in X}\frac{1}{a_2(x)}\ge h_{top}(T)$. We can assume $h_{top}(T)>0$, otherwise we have nothing to prove. For any $\mu\in M^e(X,T)$ with $h_{\mu}(T)>0$, $a_\mu(x)=\frac{1}{h_{\mu}(T,x)}=\frac{1}{h_{\mu}(T)}$ holds for $\mu$-a.e $x\in X$. We choose $x_0\in X$ with this property.

For any $0<\beta<1$, by the definition of Bowen's topological entropy and Brin-Katok's local entropy, for $\epsilon$ small enough there exits $N_{\epsilon}$ such that for $N>N_\epsilon$, one has
$$\beta h_{top}(T)\le \frac{\log r(N,\epsilon)}{N} \text{ and } -\frac{\log \mu(B_N(x_0,\epsilon))}{N}\le\frac{h_{\mu}(T)}{\beta}\le\frac{h_{top}(T)}{\beta}.$$
Hence, $-\log\mu(B_N(x_0,\epsilon))\le\frac{\log r(N,\epsilon)}{\beta^2}$ for any $N>N_{\epsilon}$.

On the other hand, for any $a>a_\mu(x_0)$ there exists $\delta_0>0$ such that $a>a_\mu(x_0,\delta)$ for any $0<\delta<\delta_0$. Then for any $0<\delta<\delta_0$and $N>N_{\epsilon}$ there exists
$$n\le -a\log \mu(B_N(x_0,\epsilon))\le \frac{a}{\beta^2}\log r(N,\epsilon)$$
satisfying $\mu(\{y\in B_N(x_0,\epsilon): \rho(T^nx_0,T^ny)>\delta\})>0$. This implies $a_2(x_0)\le \frac{a}{\beta^2}$. By taking $\beta\to 1$ and $a\to a_\mu(x_0)$, we have $a_2(x_0)\le a_{\mu}(x_0)=\frac{1}{h_{\mu}(T)}$. Therefore, by Theorem~A, one has
$$h_{top}(T)=\sup_{\mu\in M^e(X,T)}h_\mu(T)=\le \sup_{x\in X}\frac{1}{a_2(x)}.$$
This completes our proof.
\end{proof}

For a TDS $(X,T)$, the entropy map of $T$ is a map from $M(X,T)$ to $\R$ which take $\mu\in M(X,T)$ to $h_{\mu}(T)$. We have the following result.
\begin{cor}\label{cor2}Let $(X,T)$ be a TDS. If the entropy map of $T$ is upper semi-continuous, then there exists some $x\in X $ such that $\frac{1}{a_2(x)}=h_{top}(T)$.
\end{cor}
\begin{proof} Since the entropy map of $T$ is an upper semi-continuous function on the compact space $M(X,T)$, there exists $\mu \in M^e(X,T)$ such that $\int _Xh_{\mu}(T,x)d\mu=h_{\mu}(T)=h_{top}(T)$. Then $\mu( \{\ x\in X,\ h_{\mu}(T,x)\ge h_{top}(T) \}\ )>0 $.
Then,
$$h_{top}(T)\ge \frac{1}{a_2(x)}\ge \frac{1}{a_{\mu}(x)}=h_{\mu}(T,x)\ge h_{top}(T)$$
for some $x\in X $ by Theorem~A and Theorem C. This finishes our proof.
\end{proof}

\bibliographystyle{amsplain}

\end{document}